\documentclass{article}
\usepackage{amsthm}
\usepackage{amsmath}
\usepackage{epsfig}
\usepackage{amssymb}
\usepackage{latexsym}
\usepackage{enumitem}
\usepackage[T1]{fontenc}
\usepackage[utf8]{inputenc}
\usepackage{color}
\usepackage{graphicx} 

\author{
Robert Lukoťka\\Comenius University, Bratislava\\{\small\tt lukotka\@@dcs.fmph.uniba.sk}
}
\title{A Linear Bound on the Rich Flow Number for Graphs with a Given Maximum Degree}

\theoremstyle{definition}
\newtheorem{definition}{Definition}
\newtheorem{theorem}[definition]{Theorem}

\newtheorem{lemma}[definition]{Lemma}
\newtheorem{conjecture}{Conjecture}[section]

\begin{document}

\maketitle


\abstract{
A \emph{rich $k$-flow} is a nowhere-zero $k$-flow $\phi$ such that, for every pair of adjacent edges $e$ and $f$, $|\phi(e)| \neq |\phi(f)|$. A graph is \emph{rich flow admissible} if it admits a rich $k$-flow for some integer $k$. In this paper, we prove that if $G$ is a rich flow admissible graph with maximum degree $\Delta$, then $G$ admits a rich $(264\Delta - 445)$-flow.
}


\section{Introduction}
Let $A$ be an abelian group, and let $G$ be a graph, possibly with multiple edges but without loops. An \emph{$A$-flow} $\phi$ on $G$ is an assignment of values and orientations to the edges of $G$ such that, for every vertex, the sum of the values on the incoming edges equals the sum of the values on the outgoing edges. A flow $\phi$ is \emph{nowhere-zero} if $\phi(e) \neq 0$ for every edge $e$, where $0$ is the zero element of $A$. A \emph{$k$-flow} $\phi$ is a $\mathbb{Z}$-flow such that $|\phi(e)| < k$ for all edges $e$.

The study of nowhere-zero flows was initiated by Tutte~\cite{tutte1954contribution} in the 1950s. It is well-known that a graph must be bridgeless to admit a nowhere-zero flow. Tutte's $5$-flow conjecture states that every bridgeless graph admits a nowhere-zero $5$-flow~\cite{tutte1954contribution}, while Seymour's $6$-flow theorem establishes that every bridgeless graph admits a nowhere-zero $6$-flow~\cite{seymour1981nowhere}.

Recently, M\'a\v{c}ajov\'{a} and Piso\v{n}ov\'{a}~\cite{macajovapisonova} introduced a variation of the notion called \emph{rich flows}. A \emph{rich $k$-flow} is a nowhere-zero $k$-flow $\phi$ such that, for every pair of adjacent edges $e$ and $f$, $|\phi(e)| \neq |\phi(f)|$.

Of course, for a graph $G$ to admit a rich $k$-flow, it must be bridgeless. However, there is another obstruction: if two edges of a $2$-edge-cut are incident to the same vertex, they must have the same absolute flow value, violating the definition of a rich flow. Thus, such a graph cannot admit a rich $k$-flow. M\'a\v{c}ajov\'{a} and Piso\v{n}ov\'{a}~\cite{macajovapisonova} showed that these two conditions are the only obstructions. Therefore, we call a graph $G$ \emph{rich flow admissible} if it is bridgeless and no two edges forming a $2$-edge-cut are incident with a common vertex.

The \emph{rich flow number} $R(G)$ of a rich flow admissible graph $G$ is the smallest integer $k$ such that $G$ admits a rich $k$-flow.

The rich flow number is bounded from below by the chromatic index plus $1$, since the absolute flow values must form an edge coloring of $G$. Thus, the rich flow number is at least linear in the maximum degree of the graph. In this paper, we confirm that this relationship is indeed linear.

\begin{theorem}
Let $G$ be a rich flow admissible graph with maximum degree $\Delta$. Then $G$ admits a rich $(264\Delta - 521)$-flow.
\end{theorem}

Consider a graph on three vertices, where each pair of vertices is connected by $k$ edges, with $k > 1$. In such a graph $G$, all edges are adjacent, and the maximum degree is $\Delta = 2k$. However, its chromatic index is $3k$ and thus its rich flow number is at least $3k + 1$, and it is not difficult to verify that this bound is tight.

We propose the following conjecture:

\begin{conjecture}
Let $G$ be a rich flow admissible graph with maximum degree $\Delta$, where $\Delta \geq 5$. Then $G$ admits a rich $\lfloor 1.5\Delta + 1 \rfloor$-flow.
\end{conjecture}

Note that this conjecture cannot be improved simply by requiring the graph to be simple. For instance, we can replace the edges of $G$ with arbitrary subgraphs separated by a $2$-edge-cut. Flow preservation guarantees that the absolute flow values must remain the same on the edges of the $2$-edge-cut. Thus, we can replace parallel edges with subgraphs separated by a $2$-edge-cut without lowering the rich flow number.

On the other hand, if the graph is $3$-edge-connected, we believe that the bound on the rich flow number can be significantly improved:

\begin{conjecture}
Let $G$ be a $3$-edge-connected graph with maximum degree $\Delta$. Then $G$ admits a rich $(\Delta + 3)$-flow.
\end{conjecture}

This conjecture generalizes the conjecture from~\cite{macajovapisonova} for cubic graphs.

\section{The Proof}

A \emph{circuit} is a $2$-regular connected graph.
Let $G$ be a graph, and let $C_1, \dots, C_n$ be a sequence of circuits of $G$ such that non-consecutive circuits are vertex-disjoint and consecutive circuits intersect in exactly one vertex. Then, a \emph{circuit chain} of $G$ is the union of $C_1, \dots, C_n$, where $n \geq 1$. Two edges of a circuit chain are \emph{consecutive} if they are adjacent and belong to the same circuit $C_i$. Given a set of vertex-disjoint circuit chains, two edges are \emph{consecutive} if they are consecutive on one of the circuit chains in the set. A vertex of a circuit in a circuit chain is \emph{internal} if it belongs only to that circuit of the chain. A circuit chain \emph{connects} two vertices $u$ and $v$ of $G$ if $u$ and $v$ are internal vertices of $C_1$ and $C_n$, respectively, or vice versa.

We can \emph{reorient} an edge in $\phi$ by changing its orientation and negating its flow value. We take the liberty to freely reorient edges while still considering the flow to be the same $\phi$. Consider a pair of adjacent edges. Two adjacent edges are \emph{oriented consistently} if they form an oriented path of length two (or an oriented circuit of length two if both endpoints of the edges are the same). By \emph{reorienting} a pair of adjacent edges, we mean reorienting the edges so that they are oriented consistently. Given a graph $G$ and a flow $\phi$, a pair of adjacent edges is \emph{confluent} if they have the same flow value in $\phi$ after reorienting the pair. A pair of adjacent edges is \emph{contrafluent} if they have opposite flow values in $\phi$ after reorienting the pair. Note that if the equal flow values are an involution of $A$, then a pair of edges is confluent and contrafluent simultaneously.

Let $G$ be a graph, and let $\phi$ be a $(\mathbb{Z}_k \times \mathbb{Z}_2)$-flow, where $k$ is an integer. Then, a \emph{chain edge} is an edge whose value has the second coordinate equal to $1$; otherwise, the edge is a \emph{non-chain edge}. This definition is meaningful, as we will only work with $(\mathbb{Z}_k \times \mathbb{Z}_2)$-flows such that chain edges induce a set of vertex-disjoint circuit chains. If this is the case, we say that the chain edges of $\phi$ induce \emph{circuit chains of $\phi$}.

Two distinct pairs of adjacent edges are \emph{strongly intersecting} if the pairs share one edge and the edges together induce a subgraph containing a vertex of degree three.

Let $G$ be a graph, and let $A$ and $B$ be two abelian groups. By \emph{sending} a value $a \in A$ through a directed circuit $D$, we mean obtaining a flow $\phi$ on $G$ where $\phi(e) = 0$ for edges not in $D$ and $\phi(e) = a$ for edges in $D$, with edges oriented in $\phi$ as in $D$. Note that we can \emph{sum} two flows—before adding the values, we arbitrarily reorient the edges so that the orientations in the two flows are the same. Similarly, we can define the product $\phi_1 \times \phi_2$ of an $A$-flow $\phi_1$ and a $B$-flow $\phi_2$. Again, we first reorient the edges so that the orientations match, and then each edge gets the value $(a, b)$, resulting in an $(A \times B)$-flow.

We use the strategy from Seymour's proof~\cite{seymour1981nowhere} of the $6$-flow theorem. The following observation allows us to restrict Eulerian subgraphs arising in the proof to circuit chains. The lemma is a straightforward consequence of Menger's theorem~\cite[Theorem 3.3.6]{diestel}.

\begin{lemma}\label{lem:chain-connect}
Let $u$ and $v$ be two distinct vertices of a $2$-edge-connected graph $G$. Then there is a circuit chain connecting $u$ and $v$.
\end{lemma}

Now we are ready to state the main lemma of our proof.

\begin{lemma}\label{lem:building-phi}
Let $G$ be a $3$-edge-connected graph with maximum degree $\Delta$. Let $k = 8\Delta - 13$, and let $e^*$ be an edge of $G$ with a fixed orientation. Let $(a, b) \in \mathbb{Z}_k \times \mathbb{Z}_2$, where $(a, b) \neq (0, 0)$. Then $G$ admits a nowhere-zero $(\mathbb{Z}_k \times \mathbb{Z}_2)$-flow $\phi$ such that the following holds:
\begin{itemize}
\item After reorienting $e^*$ in $\phi$ to the fixed orientation, we have $\phi(e^*) = (a, b)$.
\item Chain edges of $\phi$ induce vertex-disjoint circuit chains.
\item There are no two pairs of adjacent edges confluent in $\phi$ that are strongly intersecting.
\item Every contrafluent pair of adjacent edges is consecutive on circuit chains of $\phi$.
\end{itemize}
\end{lemma}

\begin{proof}
If $b = 1$, let $C$ be a circuit containing $e^*$. If $b = 0$, let $C$ be a circuit chain connecting the end vertices of $e^*$ in $G - e^*$ (such $C$ exists due to the edge-connectivity of $G$ and Lemma~\ref{lem:chain-connect}).

We define an increasing series of $2$-edge-connected subgraphs of $G$, starting from $H_1 = C$ and finishing with $H_n = G$, as follows. Consider $H_i$:
\begin{enumerate}
\item If there is an edge $e$ of $G$ connecting two vertices of $H_i$ that is either different from $e^*$ or $e = e^*$ and it is the only edge of $G$ not in $H_i$, then set $H_{i+1} = H_i \cup e$.
\item If there is a vertex $v$ of $G$ connected to vertices of $H_i$ by two distinct edges $e_1$ and $e_2$, then set $H_{i+1} = H_i \cup v \cup e_1 \cup e_2$.
\item Otherwise, there is an edge-block $B$ of $G - V(H_i)$ such that there is at most one edge connecting $B$ with $G - V(H_i) - V(B)$ (a leaf edge-block or an isolated edge-block). Due to the connectivity of $G$, there exist two edges $e_1$ and $e_2$ between $V(H_i)$ and $V(B)$. The endpoints of $e_1$ and $e_2$ in $V(B)$ are connected in $B$ by a circuit chain $D$ due to Lemma~\ref{lem:chain-connect}. We set $H_{i+1} = H_i \cup D \cup e_1 \cup e_2$.
\end{enumerate}

Now we define a series of flows $\phi_n, \dots, \phi_1$ with the following properties for $i \in \{1, \dots, n\}$:
\begin{enumerate}[label=(\Alph*)]
\item\label{con-A} Flows $\phi_{i+1}$ and $\phi_i$ are identical when restricted to $E(G) - E(H_{i+1})$, for $i \in \{1, \dots, n-1\}$.
\item\label{con-B} Flow values are non-zero for edges from $E(G) - E(H_i)$.
\item\label{con-C} Chain edges of $\phi_i$ induce vertex-disjoint circuit chains, and all these circuit chains are in $G - V(H_i)$.
\item\label{con-E} Any two pairs of adjacent edges from $E(G) - E(H_i)$ that are confluent are not strongly intersecting.
\item\label{con-D} Every contrafluent pair of adjacent edges from $E(G) - E(H_i)$ is a consecutive pair on circuit chains of $\phi_i$.
\end{enumerate}

We start with $\phi_n(e) = (0, 0)$ for all edges $e$ of $G$, which clearly satisfies Conditions~\ref{con-A}--\ref{con-D} as $H_n = G$. To define $\phi_i$, consider how $H_{i+1}$ was created from $H_i$.

If $H_{i+1} = H_i \cup e^*$, then $H_{i+1} = G$ and $b = 0$. Let $D^*$ be an arbitrary oriented circuit of $H_i$ containing $e^*$ in the direction of the fixed orientation of $e^*$ from the lemma statement. Send $(a, 0)$ through $D^*$ and add it to $\phi_{i+1} = \phi_n$. It is easy to check that Conditions~\ref{con-A}--\ref{con-D} are satisfied. Also, from now on, the flow value of $e^*$ is as desired in all flows $\phi_i$, where $i < n$, provided that $b = 0$ due to Condition~\ref{con-A}. Note that if $b = 1$, then $e^*$ is in $C$ and thus in $H_1$.

If $H_{i+1} = H_i \cup e$, where $e \neq e^*$, then choose an oriented circuit $D^*$ of $H_{i+1}$ containing $e$. Such a circuit exists due to the connectivity of $H_i$. Send $(c, 0) \in \mathbb{Z}_k \times \mathbb{Z}_2$ through $D^*$ and add the flow to $\phi_{i+1}$ to obtain $\phi_i$. We show that it is possible to select a value of $c$ so that $\phi_i$ satisfies all conditions. Conditions~\ref{con-A} and~\ref{con-C} are clearly satisfied. Condition~\ref{con-B} forbids one value of $c$. To satisfy Conditions~\ref{con-E} and~\ref{con-D}, we choose $c$ so that no adjacent pair from $E(G) - E(H_i)$ containing $e$ is confluent or contrafluent. There are at most $2(\Delta - 3)$ such pairs (due to $2$-edge-connectivity of $H_i$). This forbids at most $4\Delta - 12$ values. The total number of forbidden values of $c$ is at most $4\Delta - 11$, which is less than $8\Delta - 15$. Thus, there is a value of $c$ that allows us to satisfy all the conditions.

If $H_{i+1} = H_i \cup v \cup e_1 \cup e_2$, then choose an oriented circuit $D^*$ of $H_{i+1}$ containing $e_1$ and $e_2$. Such a circuit exists due to the connectivity of $H_i$. Send $(c, 0) \in \mathbb{Z}_k \times \mathbb{Z}_2$ through $D^*$ and add the flow to $\phi_{i+1}$ to obtain $\phi_i$. We show that it is possible to select a value of $c$ so that $\phi_i$ satisfies all conditions. Conditions~\ref{con-A} and~\ref{con-C} are clearly satisfied. Condition~\ref{con-B} forbids up to two values of $c$. We can guarantee Conditions~\ref{con-E} and~\ref{con-D} as follows: we select $c$ so that all pairs of adjacent edges containing one of $e_1$ or $e_2$ and one edge from $E(G) - E(H_{i+1})$ are not confluent and not contrafluent. We have at most $2(\Delta - 2 + \Delta - 2)$ such pairs, which forbids at most $8\Delta - 16$ values of $c$. We also choose $c$ so that $e_1$ and $e_2$ are not contrafluent. As $k$ is odd, this forbids one value of $c$. Edges $e_1$ and $e_2$ are allowed to be confluent, as no other pair of adjacent edges from $E(G) - E(H_i)$ that are confluent contains $e_1$ or $e_2$. The total number of forbidden values of $c$ is at most $8\Delta - 15$, which is less than $8\Delta - 13$. Thus, there is a value of $c$ that allows us to satisfy all the conditions.

If $H_{i+1} = H_i \cup D \cup e_1 \cup e_2$, where $D$ consists of circuits $D_1, \dots, D_m$ (ordered so that consecutive circuits in the sequence are consecutive circuits of the chain), then choose an oriented circuit $D^*$ of $H_{i+1}$ containing $e_1$ and $e_2$. Such a circuit exists due to the connectivity of $H_i$ and $D$. Send $(c, 0) \in \mathbb{Z}_k \times \mathbb{Z}_2$ through $D^*$. Then send $(c_j, 1) \in \mathbb{Z}_k \times \mathbb{Z}_2$ through $C_j$, for $j \in \{1, \dots, m\}$. Add all these $m + 1$ flows to $\phi_{i+1}$ to get $\phi_i$. We show that with a proper choice of values $c, c_1, \dots, c_m$, we satisfy all the conditions. Conditions~\ref{con-A} and~\ref{con-C} are clearly satisfied. Condition~\ref{con-B} forbids only up to two values of $c$, since the edges from $D$ have the second coordinate equal to $1$. We can guarantee Conditions~\ref{con-E} and~\ref{con-D} as follows: for adjacent edge pairs containing $e_1$ or $e_2$, the situation is similar to the previous case. The number of forbidden values of $c$ is slightly smaller, as we have fewer adjacent pairs from $E(G) - E(H_i)$ containing $e_1$ or $e_2$. Again, it is possible that $e_1$ and $e_2$ are confluent (they may share a vertex in $H_i$). Edges from $D$ cannot be confluent or contrafluent with adjacent edges outside $D$, as they differ in the second coordinate. Pairs of adjacent edges that are consecutive on $D$ are allowed to be confluent, as such pairs are never strongly intersecting another such pair. The consecutive edges of $D$ are also allowed to be contrafluent. We can select $c_1$ arbitrarily and select the value of $c_j$, where $j \in \{2, \dots, m\}$, one by one so that in $\phi_i$, the adjacent edge pairs containing one edge of $C_{j-1}$ and one edge of $C_j$ are neither confluent nor contrafluent. This can be done, as we have just $4$ such pairs and thus just $8$ forbidden values of $c_j$. Thus, we can fulfill all the conditions.

We finish the proof. Consider first the case when $e^* \in C$. Then $C$ is a circuit. We send a value $(c, 1)$ through $C$ and add it to $\phi_1$ to get the desired flow $\phi_0$, so that $\phi(e^*)$ is $(a, b)$ in $\phi_0$ after reorienting $e^*$ to the fixed orientation. We show that $\phi_0$ fulfills the requirements of the lemma. The flow $\phi_0$ is nowhere zero outside $C$ due to Condition~\ref{con-B} for $\phi_1$ and is non-zero on $C$, as the second coordinate is non-zero on $C$. The edge $e^*$ has the desired flow value. Also, clearly, chain edges form a set of vertex-disjoint circuit chains due to Condition~\ref{con-C} for $\phi_1$. Edges of $C$ can be contrafluent only with their adjacent edges in $C$, as they differ in the second coordinate from other edges due to Condition~\ref{con-C} for $\phi_1$, which is allowed by the lemma statement. Edges outside of $C$ can be contrafluent only if they are consecutive pairs on circuit chains in $\phi_1$ by Condition~\ref{con-D}, which means that they are consecutive pairs on circuit chains in $\phi_0$. Consecutive pairs on $C$ may be confluent, but they are not strongly intersecting. An edge of $C$ cannot be confluent with an adjacent edge from $E(G) - E(C)$, as they differ in the second coordinate due to Condition~\ref{con-C} for $\phi_1$. Remaining confluent edges are not strongly intersecting due to Condition~\ref{con-E} for $\phi_1$.

In the case where $e^* \notin C$, we have a circuit chain $C_1, \dots, C_m$. Send $(c_j, 1) \in \mathbb{Z}_k \times \mathbb{Z}_2$ through $C_j$, for $j \in \{1, \dots, m\}$. Add all these $m$ flows to $\phi_1$ to get $\phi_0$. We select $c_1$ arbitrarily. We select $c_j$, where $j \in \{2, \dots, m\}$, so that adjacent edges, where one edge is from $C_{j-1}$ and the second edge is from $C_j$, are neither confluent nor contrafluent. We can do this, as we forbid just $8$ values of $c_j$, which is less than $8\Delta - 13$. We name the resulting flow $\phi_0$ and show that it fulfills the lemma requirements. The flow $\phi_0$ is nowhere zero outside $C$ due to Condition~\ref{con-B} for $\phi_1$ and is non-zero on $C$, as the second coordinate is non-zero on $C$. The edge $e^*$ has the desired flow value, as it has the desired flow value in $\phi_{n-1}$ and thus also in $\phi_1$ by Condition~\ref{con-A}. Also, clearly, chain edges form a set of vertex-disjoint circuit chains due to Condition~\ref{con-C} for $\phi_1$. Edges of $C$ cannot be confluent or contrafluent with adjacent edges of $E(G) - E(C)$, as they differ by the second coordinate due to Condition~\ref{con-C}. This limits confluency and contrafluency to $C$ and $E(G) - E(C)$ separately. In $E(G) - E(C)$, the conditions are satisfied due to Conditions~\ref{con-E} and~\ref{con-D} for $\phi_1$. For edges in $C$, our choice limits confluency and contrafluency to consecutive edge pairs of $C$, but such confluent or contrafluent pairs are allowed by the lemma statement.
\end{proof}

\begin{lemma}\label{lem:2ec}
Let $G$ be a rich flow admissible graph with maximum degree $\Delta$.
Let $k = 8\Delta - 13$.
Then $G$ has a nowhere-zero $(\mathbb{Z}_k \times \mathbb{Z}_2)$-flow $\phi$ such that the following holds:
\begin{itemize}
\item Chain edges induce vertex-disjoint circuit chains.
\item There are no two pairs of neighboring edges confluent in $\phi$ that are strongly intersecting.
\item Every contrafluent pair of neighboring edges is consecutive on circuit chains of $\phi$.
\end{itemize}
\end{lemma}

\begin{proof}
Consider the smallest counterexample $G$ to this lemma. The graph $G$ is not $3$-edge-connected due to Lemma~\ref{lem:building-phi}.

Consider a $2$-edge-cut $\{e_1, e_2\}$ separating two subgraphs $G_1$ and $G_2$ of $G$ such that $G_2$ is as small as possible. The edges $e_1$ and $e_2$ have two distinct endvertices, both in $G_1$ and in $G_2$; otherwise, $G$ would not be rich flow admissible. Let $u_1$ and $u_2$ be the endvertices of $e_1$ and $e_2$ in $G_1$, respectively, and let $v_1$ and $v_2$ be the endvertices of $e_1$ and $e_2$ in $G_2$, respectively. Add an edge $u_1u_2$ to $G_1$ to obtain $G'_1$ and an edge $v_1v_2$ to $G_2$ to obtain $G'_2$.

We show that $G'_1$ is rich flow admissible. Indeed, the only cuts in $G'_1$ that are not in $G$ are those containing $u_1u_2$. However, if we replace $u_1u_2$ in the cut with $e_1$ or $e_2$, we get a cut. Thus, $u_1u_2$ cannot be a bridge of $G'_1$. If there were a $2$-edge-cut containing two edges incident with the same vertex, and one of the edges is $u_1u_2$, then the vertex must be $u_1$ or $u_2$. But then we obtain a $2$-edge-cut violating the rich flow admissibility of $G$ by replacing $u_1u_2$ with $e_1$ if the common vertex was $u_1$ and with $e_2$ if the common vertex was $u_2$.

Now we show that, due to the minimality of $G_2$, the graph $G'_2$ is $3$-edge-connected. Indeed, if a $2$-edge-cut does not contain $v_1v_2$, then it is also a $2$-edge-cut of $G$ and separates a subgraph smaller than $G_2$, contradicting the choice of $G_2$. If the $2$-edge-cut contains $v_1v_2$, then we can replace $v_1v_2$ with $e_1$ to obtain a $2$-edge-cut that separates a subgraph of $G$ smaller than $G_2$, again contradicting the choice of $G_2$.

We apply the induction hypothesis on $G_1$ and obtain a flow $\phi_1$. We use Lemma~\ref{lem:building-phi} on $G_2$ and set $e^*$ to be $v_1v_2$, and we set $(a, b)$ and the orientation of $e^*$ so that it matches the value and orientation of $u_1v_2$. We obtain a flow $\phi_2$. Now it is straightforward to combine the rich flows $\phi_1$ and $\phi_2$ into a rich flow $\phi$ of $G$. For edges within $G_1$, we take orientation and values from $\phi_1$; for edges within $G_2$, we take orientation and values from $\phi_2$; and for $e_1$ and $e_2$, we take orientation and flow value according to $e^*$. It is clear that this is the desired rich flow.
\end{proof}

\begin{lemma}\label{lem:remove-confluent}
Let $G$ be a rich flow admissible graph. Let $\mathcal{P}$ be a set of pairs of neighboring edges of $G$ such that no two pairs in the set are strongly intersecting. Then there exists a nowhere-zero $\mathbb{Z}_6$-flow $\phi$ on $G$ such that no pair from $\mathcal{P}$ is confluent in $\phi$.
\end{lemma}

\begin{proof}
    For each pair $p \in \mathcal{P}$, we fix a vertex contained in both edges of $p$ and denote it by $a(p)$. If the pair $p$ shares two vertices, we choose $a(p)$ arbitrarily. We also introduce a new vertex $b(p)$, distinct from all vertices in $G$ and from all other vertices $b(p')$ for $p' \in \mathcal{P}$.

    We construct an auxiliary graph $H$ as follows. The vertex set of $H$ is given by
    $V(H) = V(G) \cup \{ b(p) \mid p \in \mathcal{P} \}$.
    The edge set of $H$ contains edges $a(p)b(p)$, for all $p\in \mathcal{P}$, as well as edges corresponding with the edges of $G$ defined as follows:
    \begin{itemize}
        \item For each edge $e = uv$ in $G$ that does not appear in any pair of $\mathcal{P}$, we add an edge $uv$ to $H$.
        \item If $e = uv$ appears in exactly one pair $p \in \mathcal{P}$, then we replace the vertex $a(p)$ in the edge with the vertex $b(p)$ and add the modified edge to $H$.
        \item If $e$ appears in two distinct pairs $p_1$ and $p_2$, then, as $p_1$ and $p_2$ do not strongly intersect, $a(p_1) \neq a(p_2)$. Replace $a(p_1)$ with $b(p_1)$ and $a(p_2)$ with $b(p_2)$ and add the edge $b(p_1)b(p_2)$ to $H$.
        \item No edge appears in three or more pairs of $\mathcal{P}$ as this implies two strongly intersecting pairs. 
    \end{itemize}

    Observe that contracting the edges $a(p)b(p)$ for all $p \in \mathcal{P}$ in $H$ yields a graph isomorphic to $G$. Additionally, each vertex $b(p)$ for $p \in \mathcal{P}$ has degree three in $H$.

    We claim that $H$ is bridgeless. Only edges $a(p)b(p)$, for $p \in \mathcal{P}$, can possibly be bridges of $H$, as all other edges have a corresponding edge in $G$ that would then be a bridge of $G$. Suppose, for contradiction, that an edge $a(p)b(p)$ is a bridge in $H$. Then, the edges in $p$ would form a 2-edge-cut in $G$, contradicting the rich flow admissibility of $G$. Therefore, $H$ is bridgeless.

    By Seymour's 6-flow theorem, $H$ admits a nowhere-zero $\mathbb{Z}_6$-flow. Contracting the edges $a(p)b(p)$, for all $p \in \mathcal{P}$, in $H$ while keeping the flow yields a nowhere-zero $\mathbb{Z}_6$-flow in $G$. Since the vertices $b(p)$, $p \in \mathcal{P}$, have degree three $H$, no two edges incident with these vertices are confluent in the flow on $H$ and thus the edges in $p$ are also not confluent in the contracted flow in $G$, as desired.
\end{proof}

\begin{theorem}\label{thm:general}
Let $G$ be a rich flow admissible graph with maximum degree $\Delta$. Then $G$ has a rich $(264\Delta - 445)$-flow.
\end{theorem}

\begin{proof}
Take a $\mathbb{Z}_k$-flow $\phi'_1$ and $\mathbb{Z}_2$-flow $\phi'_2$ guaranteed by Lemma~\ref{lem:2ec}, and a $\mathbb{Z}_6$-flow $\phi'_3$ guaranteed by Lemma~\ref{lem:remove-confluent} such that every pair of confluent edges in $\phi'_1 \times \phi'_2$ is not confluent in $\phi'_3$.

Convert $\phi'_1$, $\phi'_2$, and $\phi'_3$ into a $k$-flow $\phi_1$, a $2$-flow $\phi_2$, and a $6$-flow $\phi_3$, respectively, keeping orientations the same and keeping the flow values equal modulo $k$, $2$, and $6$ (see \cite[Proof of Theorem 6.3.3]{diestel}).

We create a new flow $\phi = \phi_3 + 11(\phi_2 + 3\phi_1)$. This is a $(264\Delta - 445)$-flow, as the maximum possible value is $5 + 11(1 + 3(8\Delta - 14)) = 264\Delta - 446$. We show that the resulting flow $\phi$ is rich.

Suppose, for contradiction, that there is a pair of neighboring edges $e$ and $f$ with the same absolute flow value $|\phi(e)| = |\phi(f)|$. Reorient $e$ and $f$ so they are oriented consistently. Also reorient $e$ and $f$ in $\phi_1$, $\phi_2$, $\phi_3$, $\phi'_1$, $\phi'_2$, and $\phi'_3$ so that the orientation of the edges is the same as in $\phi$.

Suppose first that $\phi(e) = \phi(f)$. Since all different flow values in $\phi_3$ are distinct modulo $11$, we have $\phi_3(e) = \phi_3(f)$. Therefore, $(\phi_2 + 3\phi_1)(e) = (\phi_2 + 3\phi_1)(f)$. But all flow values in $\phi_2$ are distinct modulo $3$, so $\phi_2(e) = \phi_2(f)$ and $\phi_1(e) = \phi_1(f)$. Since we preserved values modulo $k$, $2$, and $6$ when constructing $\phi_1$, $\phi_2$, and $\phi_3$, we also have $\phi'_1(e) = \phi'_1(f)$, $\phi'_2(e) = \phi'_2(f)$, and $\phi'_3(e) = \phi'_3(f)$. But then $e$ and $f$ are confluent in $\phi'_1 \times \phi'_2$, and they cannot be confluent in $\phi'_3$ due to the choice of $\phi'_3$ by Lemma~\ref{lem:remove-confluent}, a contradiction.

If $\phi(e) = -\phi(f)$, then similarly, $\phi'_1(e) = -\phi'_1(f)$, $\phi'_2(e) = -\phi'_2(f)$, and $\phi'_3(e) = -\phi'_3(f)$. Thus, $e$ and $f$ are contrafluent in $\phi'_1 \times \phi'_2$. Therefore, they are consecutive on circuit chains of $\phi'_1 \times \phi'_2$. But in a $2$-flow of a circuit chain, consecutive edges must be confluent, which implies $\phi_2(e) = \phi_2(f)$ is $1$ or $-1$, a contradiction with $\phi'_2(e) = -\phi'_2(f)$.
\end{proof}

The bound we present is certainly far from tight, and we did not attempt to optimize the constant $264$, as we believe that a different approach is necessary to significantly reduce it.

\section{Concluding Remarks}

Examining rich flows may also be interesting when we do not require the rich flow to be nowhere-zero. The proof of Theorem~\ref{thm:general} can be adjusted to improve the constant term under these relaxed conditions. Note that such flows exist even in some graphs with bridges.

Another interesting problem arises if, instead of forbidding both confluent and contrafluent edges, we forbid only one of the two situations. If we forbid only contrafluent edges, then using the argumentation from the proof of Theorem~\ref{thm:general}, one can improve the constant at $\Delta$ by a factor of two. The rich flow number will still be linearly dependent on $\Delta$, as at most two edges incident with a vertex can have the same absolute flow value. Note that in this case, a graph can have two edges of a $2$-edge-cut incident with a common vertex.

On the other hand, if we forbid only confluent edges, it is possible that the relationship between the rich flow number and $\Delta$ may be sublinear. We believe the following:

\begin{conjecture}
Every rich flow admissible graph with maximum degree $\Delta$ has a nowhere-zero $O(\ln(\Delta))$-flow that contains no pair of confluent edges.
\end{conjecture}

\bibliographystyle{plain}
\bibliography{main}

\end{document}